\documentclass[12pt]{amsart}
\usepackage{amsmath,amsthm,amsfonts,amssymb,latexsym}

\usepackage{hyperref}

\usepackage{enumerate}
\usepackage[shortlabels]{enumitem}

\begin{document}

\theoremstyle{plain}

\newtheorem{thm}{Theorem}[section]
\newtheorem{lem}[thm]{Lemma}
\newtheorem{Problem B}[thm]{Problem B}

\newtheorem{pro}[thm]{Proposition}
\newtheorem{cor}[thm]{Corollary}
\newtheorem{que}[thm]{Question}
\newtheorem{rem}[thm]{Remark}
\newtheorem{defi}[thm]{Definition}

\newtheorem*{thmA}{Theorem A}
\newtheorem*{corB}{Corollary B}
\newtheorem*{thmC}{Theorem C}
\newtheorem*{thmD}{Theorem D}

\newtheorem*{thmAcl}{Main Theorem$^{*}$}
\newtheorem*{thmBcl}{Theorem B$^{*}$}

\newcommand{\Maxn}{\operatorname{Max_{\textbf{N}}}}
\newcommand{\Syl}{\operatorname{Syl}}
\newcommand{\dl}{\operatorname{dl}}
\newcommand{\Con}{\operatorname{Con}}
\newcommand{\cl}{\operatorname{cl}}
\newcommand{\Stab}{\operatorname{Stab}}
\newcommand{\Aut}{\operatorname{Aut}}
\newcommand{\Ker}{\operatorname{Ker}}
\newcommand{\fl}{\operatorname{fl}}
\newcommand{\Irr}{\operatorname{Irr}}
\newcommand{\SL}{\operatorname{SL}}
\newcommand{\FF}{\mathbb{F}}
\newcommand{\NN}{\mathbb{N}}
\newcommand{\N}{\mathbf{N}}
\newcommand{\bfC}{\mathbf{C}}
\newcommand{\bfO}{\mathbf{O}}
\newcommand{\bfF}{\mathbf{F}}

\renewcommand{\labelenumi}{\upshape (\roman{enumi})}

\newcommand{\PSL}{\operatorname{PSL}}
\newcommand{\PSU}{\operatorname{PSU}}

\providecommand{\V}{\mathrm{V}}
\providecommand{\E}{\mathrm{E}}
\providecommand{\ir}{\mathrm{Irr_{rv}}}
\providecommand{\Irrr}{\mathrm{Irr_{rv}}}
\providecommand{\re}{\mathrm{Re}}

\numberwithin{equation}{section}
\def\irrp#1{{\rm Irr}_{p'}(#1)}

\def\ibrrp#1{{\rm IBr}_{\Bbb R, p'}(#1)}
\def\Z{{\mathbb Z}}
\def\C{{\mathbb C}}
\def\Q{{\mathbb Q}}
\def\irr#1{{\rm Irr}(#1)}
\def\ibr#1{{\rm IBr}(#1)}
\def\irrp#1{{\rm Irr}_{p^\prime}(#1)}
\def\irrq#1{{\rm Irr}_{q^\prime}(#1)}
\def \c#1{{\cal #1}}
\def\cent#1#2{{\bf C}_{#1}(#2)}
\def\syl#1#2{{\rm Syl}_#1(#2)}
\def\nor{\triangleleft\,}
\def\oh#1#2{{\bf O}_{#1}(#2)}
\def\Oh#1#2{{\bf O}^{#1}(#2)}
\def\zent#1{{\bf Z}(#1)}
\def\det#1{{\rm det}(#1)}
\def\ker#1{{\rm ker}(#1)}
\def\norm#1#2{{\bf N}_{#1}(#2)}
\def\alt#1{{\rm Alt}(#1)}
\def\iitem#1{\goodbreak\par\noindent{\bf #1}}
   \def \mod#1{\, {\rm mod} \, #1 \, }
\def\sbs{\subseteq}

\def\gc{{\bf GC}}
\def\ngc{{non-{\bf GC}}}
\def\ngcs{{non-{\bf GC}$^*$}}
\newcommand{\notd}{{\!\not{|}}}
\newcommand{\Out}{{\mathrm {Out}}}
\newcommand{\Mult}{{\mathrm {Mult}}}
\newcommand{\Inn}{{\mathrm {Inn}}}
\newcommand{\IBR}{{\mathrm {IBr}}}
\newcommand{\IBRL}{{\mathrm {IBr}}_{\ell}}
\newcommand{\IBRP}{{\mathrm {IBr}}_{p}}
\newcommand{\ord}{{\mathrm {ord}}}
\def\id{\mathop{\mathrm{ id}}\nolimits}
\renewcommand{\Im}{{\mathrm {Im}}}
\newcommand{\Ind}{{\mathrm {Ind}}}
\newcommand{\diag}{{\mathrm {diag}}}
\newcommand{\soc}{{\mathrm {soc}}}
\newcommand{\End}{{\mathrm {End}}}
\newcommand{\sol}{{\mathrm {sol}}}
\newcommand{\Hom}{{\mathrm {Hom}}}
\newcommand{\Mor}{{\mathrm {Mor}}}
\newcommand{\Mat}{{\mathrm {Mat}}}
\def\rank{\mathop{\mathrm{ rank}}\nolimits}
\newcommand{\Tr}{{\mathrm {Tr}}}
\newcommand{\tr}{{\mathrm {tr}}}
\newcommand{\Gal}{{\it Gal}}
\newcommand{\Spec}{{\mathrm {Spec}}}
\newcommand{\ad}{{\mathrm {ad}}}
\newcommand{\Sym}{{\mathrm {Sym}}}
\newcommand{\Char}{{\mathrm {char}}}
\newcommand{\pr}{{\mathrm {pr}}}
\newcommand{\rad}{{\mathrm {rad}}}
\newcommand{\abel}{{\mathrm {abel}}}
\newcommand{\codim}{{\mathrm {codim}}}
\newcommand{\ind}{{\mathrm {ind}}}
\newcommand{\Res}{{\mathrm {Res}}}
\newcommand{\Lie}{{\mathrm {Lie}}}
\newcommand{\Ext}{{\mathrm {Ext}}}
\newcommand{\Alt}{{\mathrm {Alt}}}
\newcommand{\AAA}{{\sf A}}
\newcommand{\SSS}{{\sf S}}
\newcommand{\CC}{{\mathbb C}}
\newcommand{\CB}{{\mathbf C}}
\newcommand{\RR}{{\mathbb R}}
\newcommand{\QQ}{{\mathbb Q}}
\newcommand{\ZZ}{{\mathbb Z}}
\newcommand{\bfN}{{\mathbf N}}
\newcommand{\bfZ}{{\mathbf Z}}
\newcommand{\EE}{{\mathbb E}}
\newcommand{\PP}{{\mathbb P}}
\newcommand{\cG}{{\mathcal G}}
\newcommand{\cH}{{\mathcal H}}
\newcommand{\cQ}{{\mathcal Q}}
\newcommand{\GA}{{\mathfrak G}}
\newcommand{\cT}{{\mathcal T}}
\newcommand{\cS}{{\mathcal S}}
\newcommand{\cR}{{\mathcal R}}
\newcommand{\cB}{{\mathcal B}}
\newcommand{\GCD}{\GC^{*}}
\newcommand{\TCD}{\TC^{*}}
\newcommand{\FD}{F^{*}}
\newcommand{\GD}{G^{*}}
\newcommand{\HD}{H^{*}}
\newcommand{\GCF}{\GC^{F}}
\newcommand{\TCF}{\TC^{F}}
\newcommand{\PCF}{\PC^{F}}
\newcommand{\GCDF}{(\GC^{*})^{F^{*}}}
\newcommand{\RGTT}{R^{\GC}_{\TC}(\theta)}
\newcommand{\RGTA}{R^{\GC}_{\TC}(1)}
\newcommand{\Om}{\Omega}
\newcommand{\eps}{\epsilon}
\newcommand{\al}{\alpha}
\newcommand{\chis}{\chi_{s}}
\newcommand{\sigmad}{\sigma^{*}}
\newcommand{\PA}{\boldsymbol{\alpha}}
\newcommand{\gam}{\gamma}
\newcommand{\lam}{\lambda}
\newcommand{\la}{\langle}
\newcommand{\ra}{\rangle}
\newcommand{\hs}{\hat{s}}
\newcommand{\htt}{\hat{t}}
\newcommand{\tn}{\hspace{0.5mm}^{t}\hspace*{-0.2mm}}
\newcommand{\ta}{\hspace{0.5mm}^{2}\hspace*{-0.2mm}}
\newcommand{\tb}{\hspace{0.5mm}^{3}\hspace*{-0.2mm}}
\def\skipa{\vspace{-1.5mm} & \vspace{-1.5mm} & \vspace{-1.5mm}\\}
\newcommand{\tw}[1]{{}^#1\!}
\renewcommand{\mod}{\bmod \,}

\marginparsep-0.5cm

\renewcommand{\thefootnote}{\fnsymbol{footnote}}
\footnotesep6.5pt

\title{Finite Groups with Odd Sylow Normalizers}

\author{Robert M. Guralnick}
\address{Department of Mathematics, University of Southern California, 3620 S. Vermont Ave., Los Angeles, CA 90089, USA}
\email{guralnic@usc.edu}
\author{Gabriel Navarro}
\address{Departament d'\`Algebra, Universitat de Val\`encia, 46100 Burjassot,
Val\`encia, Spain}
\email{gabriel.navarro@uv.es}
\author{Pham Huu Tiep}
\address{Department of Mathematics, University of Arizona, Tucson, AZ 85721, USA}
\email{tiep@math.arizona.edu}

\thanks{The first author gratefully acknowledges the support of the NSF grant DMS-1302886.}
 
\thanks{The research of the second author is supported by the Prometeo/Generalitat 
Valenciana,
Proyectos MTM2013-40464-P.  He would like to express his gratitude to the Mathematics Department
of the University of Southern California  where part of the present work was completed
 for its warm hospitality.} 
\thanks{The third author was partially supported by the NSF grant DMS-1201374 and the Simons 
  Foundation Fellowship 305247. Part of this work was done while the third author visited the Institute for 
  Advanced Study (Princeton, NJ). It is a pleasure to thank Peter Sarnak and the Institute for 
  generous hospitality and stimulating environment.}
\thanks{The authors are grateful to Gunter Malle for helpful comments on the paper.}



\begin{abstract}
We determine the non-abelian composition factors of the finite groups
with Sylow normalizers of odd order. As a consequence, among others,
we prove the McKay conjecture and the Alperin weight conjecture for these 
groups at these primes.
\end{abstract}

\maketitle

\section{Introduction}
 
 Suppose that $G$ is a finite group, $p$ is an odd prime number, and $P$ is a Sylow
 $p$-subgroup of $G$.
 In \cite{GMN}, the first two authors and G. Malle 
determined the non-abelian composition factors
 of $G$ if $P=\norm GP$. This led to proving a strong form of the McKay conjecture in \cite{N} and \cite{NTV}
 for these groups at the prime $p$. Now, instead of assuming that $\norm GP/P$ is trivial, we assume that it has odd order.
 Although in this case
 the structure of $G$ can be fairly complicated,  we nevertheless
 are able to have control on the non-abelian composition factors of $G$. 
 
 \begin{thmA}
 Let $G$ be a finite group, let $p$ be a prime, and $P \in \syl pG$.
 Assume that $|\norm GP|$ is odd. If $S$ is a non-abelian composition
 factor of $G$, then  $|S|$ is divisible by $p$, and
 either $S$ has cyclic Sylow $p$-subgroups or $S=PSL_2(q)$,
 for some power $q=p^f \equiv 3 (\mod 4)$. 
 \end{thmA}
 
There are too many almost simple groups with odd Sylow normalizers
to be listed. However, once Theorem A is proved,
we can apply the results in
\cite{IMN2}, \cite{KS}, \cite{S1} and \cite{S2} to establish the main counting conjectures for
groups with odd Sylow normalizers.
(Recall that $\irrp G$ is the set of the irreducible complex characters of $G$
with degree not divisible by $p$.)

\begin{corB}
Let $G$ be a finite group, let $p$ be a prime and $P \in \syl pG$.
If $|\norm GP|$ is odd, then 
$$|\irrp G|=|\irrp{\norm GP}| \, .$$
Furthermore, the Alperin-McKay conjecture and the blockwise Alperin weight conjecture hold
true for $G$, for the prime $p$.
\end{corB}
 
In view of the recent proof by G. Malle and B. Sp\"ath \cite{MS} of the McKay conjecture for the prime $2$, 
we can now state that the McKay conjecture holds for the prime $p$ and for all finite groups $G$ whenever 
$|\bfN_G(P)/P|$ is odd for $P \in \Syl_p(G)$.

\medskip
Contrary to the case where $\norm GP=P$ (see \cite{N} and \cite{NTV}), there does not seem to
exist a canonical choice-free
bijection $\irrp G \rightarrow \irrp{\norm GP}$ if $|\norm GP|$ is odd. (If $G$ is solvable,
A. Turull did find a canonical bijection in   \cite{T}.) At least, we know for a fact that
there cannot exist
a bijection that commutes with complex conjugation.
In this case,  the trivial character would be the only real-valued irreducible character
of $G$ of $p'$-degree, and this is plainly false:
it was proved in \cite{IMN1} that for $p>3$ every non-solvable group
has a non-trivial real-valued irreducible character of $p'$-degree. 
On the other hand, we prove the following. 
\begin{thmC}
Let $G$ be a finite group and let $p$ be a prime, let  $P \in \syl pG$, and assume that
$|\norm GP|$ is odd. If $\chi \in \irrp G$ is real-valued, then $N \leq \ker\chi$ for every
solvable $N \nor G$.
\end{thmC}

Finally, we remark
that it is a basic problem in character theory to investigate how the character table of a finite group reflects
(and is reflected by) its group theoretical properties. We wished to be able to detect from the character table
of $G$ whether or not $|\norm GP|$ is odd, but this seems to be an elusive problem (at least,
if $P$ is not abelian).
When finding properties which are detectable in the character table 
(as Theorem C), it is easy to prove that if $|\norm GP|$ has odd order, then
the only real class of $\bar G$ with $p'$-size is the identity, for every factor group $\bar G$ of $G$.
This condition characterizes that $|\norm GP|$ is odd in many groups, such as solvable groups, 
or groups with an abelian Sylow $p$-subgroups. 
(We write down proofs of these elementary results in Section 4 below.)
A general if and only if condition, however, remains to be discovered.

 \section{Proofs of Theorem A and Corollary B}

Our results rely on the following result which will be proved in the next section:

 \begin{thm}\label{main:simple}
  Suppose that $G$ is an almost simple group with socle $S$.
  Let $P \in \syl pG$ such that $G=SP$, $p$ divides $|S|$, and $|\norm GP|$ is odd.
  Then either $S$ has cyclic Sylow $p$-subgroups or $S=PSL_2(q)$ for some power
  $q = p^f \equiv 3 (\mod 4)$.
  \end{thm}

First we prove Theorem A, which we restate:

 \begin{thm}
 Let $G$ be a finite group, let $p$ be a prime, and $P \in \syl pG$.
 Assume that $|\norm GP|$ is odd. If $S$ is a non-abelian composition
 factor of $G$, then  $|S|$ is divisible by $p$, and
 either $S$ has cyclic Sylow $p$-subgroups or $S=PSL_2(q)$,
 for some power $q = p^f \equiv 3 (\mod 4)$. 
 \end{thm}
 
 \begin{proof}
We argue by induction on $|G|$.
Since the hypotheses are inherited by factor groups,
 it is enough to show that if $N$ is a minimal  non-abelian normal
 subgroup of $G$, and $S \nor N$ is simple, then
 $|S|$ is divisible by $p$ and either $S$ has cyclic Sylow $p$-subgroups or $S=PSL_2(q)$,
 for some prime power $q = p^f \equiv 3 (\mod 4)$. First of all, if $|N|$ is not divisible by $p$,
 then $\cent NP=\norm NP$ has odd order, and therefore $N$ is solvable
 by a standard application of the Classification of Finite Simple Groups to coprime action
 (see Theorem 3.4 of \cite{IMN1}).
 Hence, we may assume that $|S|$ has order divisible by $p$.
 Now, $NP$ has an odd order Sylow normalizer, so we may assume that $NP=G$.
 Let $Q=P\cap N$. We have that $\norm GP/P$ is isomorphic to
 $\cent{\norm NQ/Q}P$.
 
 We can write $N=S^{x_1} \times \ldots \times S^{x_t}$, where  
 $U=\{x_1, \ldots, x_t\}$ is a complete set of representatives of right cosets
 of $\norm PS$ in $P$. Set $R=P\cap S=Q\cap S$, $H=S\norm P S$
 and $P_1=\norm PS \in \syl p H$. We have that $\norm H{P_1}/P_1$ is isomorphic
 to $\cent {\norm SR/R}{P_1}$.
 If $xR \in \cent {\norm SR/R}{P_1}$ is an involution,
 then it is straightforward to check that 
 $$z:=\prod_{u \in U} x^u \in \norm NQ$$
  and that $zQ$ is an involution centralized by $P$. But this is impossible since
  $\cent{\norm NQ/Q}P$ has odd order by hypothesis. Hence,
  we conclude that $H$ is an almost simple group with socle $S$ and $H/S$ a $p$-group
  having a Sylow $p$-subgroup $P_1$ with normalizer of odd order. Now we apply Theorem \ref{main:simple}.
  \end{proof}

Once Theorem A is completed, Corollary B, which we restate, easily follows.
 
 \begin{cor}
 Let $G$ be a finite group, let $p$ be a prime and $P \in \syl pG$.
If $|\norm GP|$ is odd, then 
$$|\irrp G|=|\irrp{\norm GP}| \, .$$
Moreover, the Alperin-McKay conjecture and the blockwise Alperin weight conjecture hold
true for $G$, for the prime $p$.
  \end{cor}
  
  \begin{proof}
 It is well known (see e.g. \cite[Corollary (11.21)]{Is})
 that if a prime $r$ divides the order of the Schur multiplier of a non-abelian simple group 
 $S$ then the Sylow $r$-subgroups of $S$ are non-cyclic. In particular, 
 if Sylow $p$-subgroups of $S$ are cyclic, then the same
 holds for its universal cover. Hence, it follows from \cite[Theorem 1.1, Corollary 6.9, Corollary 7.1]{KS} 
 that simple groups with cyclic Sylow $p$-subgroups satisfy the inductive
  Alperin-McKay condition
  and the inductive condition for the blockwise Alperin weight conjecture (for the prime $p$).
 Next, it is proved in  \cite{S1} and \cite{S2}
 that $PSL_2(q)$ with $q =p^f \equiv 3 (\mod 4)$
  also satisfies both inductive conditions 
  (although the case $q=3^f >9$ is not explicitly stated there, but follows with the
  same argument). 
  Now, if $G$ is a finite group with odd Sylow normalizers, it follows 
  from Theorem A, elementary group theory, and the classification
  of the subgroups of $PSL_2(q)$ with $q \equiv 3 (\mod 4)$, that every non-abelian composition factor
  of order divisible by $p$ of every subgroup of $G$ either has cyclic Sylow $p$-subgroups, or 
  is isomorphic to some $PSL_2(q')$ with $q' \equiv 3 (\mod 4)$. 
  Hence the main results of \cite{IMN2}, \cite{S1} and \cite{S2} apply.
  \end{proof}
 
 \section{Proof of Theorem \ref{main:simple}}
 
 We begin with a simple observation:
 
 \begin{lem}\label{real}
 Keep the notation and hypothesis of Theorem \ref{main:simple}. Then $S \cap P$ contains a non-real 
 $p$-central element $z$ of order $p$. In fact, every $p$-element $1 \neq y \in S$ that is $p$-central in 
 $G$ must be non-real.
 \end{lem}
 
 \begin{proof}
 By assumption, $1 \neq Q := P \cap S \nor P$, and so we can choose $1 \neq z \in \Omega_1(\bfZ(P)) \cap Q$. Since $p > 2$
 and $2 \nmid |\bfN_G(P)|$, $z$ is not real in $\bfN_G(P)$, whence $z$ is not real in $G$ (and so in $S$) by 
 Burnside's fusion control lemma. The second conclusion also follows by applying the same argument to $y$.
 \end{proof}
 
 \begin{cor}\label{simple1}
 Suppose that the simple group $S$ in Theorem \ref{main:simple} is of Lie type defined over $\FF_q$.
 
 \begin{enumerate}[\rm (i)]
 
 \item If $p \nmid q$, then $S$ 
 can only be possibly of types $A_n$ or $\tw2 A_n$ with $n \geq 2$, $D_n$ or $\tw2 D_n$ with $2 \nmid n \geq 5$,
 $E_6$ or $\tw2 E_6$. 
 
 \item Suppose $p|q$. Then either $q \equiv 3 (\mod 4)$, or $p \in \{3,5\}$ and $S$ is an exceptional group of Lie type. 
 Furthermore, $S$ is not of types $D_4$, $\tw 2D_4$, $\tw3 D_4$.
 \end{enumerate}
 \end{cor}
 
 \begin{proof}
Recall that $p$ is odd.

 (i) Assume the contrary. By \cite[Proposition 3.1]{TZ} all semisimple elements in $G$ are real. In particular, the element $z$ 
 described in Lemma \ref{real} is real, a contradiction. Note that our argument also applies to $\tw2 F_4(2)'$.
 
 (ii) Assume the contrary. By Lemma 2.2 and Theorem 1.4 of \cite{TZ}, all unipotent elements in $G$ are real. Hence the element $z$ 
 described in Lemma \ref{real} is real, again a contradiction.
 \end{proof}
 
\begin{lem}\label{alt}
Let $S = \AAA_n$ be an alternating group.
 
\begin{enumerate}[\rm(i)]
 
\item If $n \geq p+2$, then any element $t$ of order $p$ in $S$ is real. In particular,
 $|\bfN_S(Q)|$ is even for $Q \in \Syl_2(S)$.

\item Similarly, if $n \geq 2$ and $P \in \Syl_p(\SSS_n)$, then $|\bfN_{\SSS_n}(P)|$ is even.
 
\item Theorem \ref{main:simple} holds if $S = \AAA_n$ is an alternating group.
\end{enumerate}
\end{lem}
 
\begin{proof}
(i) Write $t$ as a product of $b$ disjoint cycles. Then $t$ has $a := n-bp$ fixed points. As $n \geq p+2$, we have that $a \geq 2$ or $b \geq 2$.
 Hence $t$ is centralized by an odd permutation, and so it is real in $S$. Arguing as in the proof of Lemma \ref{real}, we see that
 $|\bfN_S(Q)|$ is even.
 
 (ii) The statement is obvious if $n < p$. If $n \geq p$, then any element (of order $p$) in $\SSS_n$ is real, so we can argue as in the proof
 of Lemma \ref{real}.
 
 (iii) Applying (i) to the element $z$ obtained in Lemma \ref{real} we see that $n = p$ or $p+1$.
 In particular, the Sylow $p$-subgroups of $S$ are cyclic. (In fact, one can 
also show that $p \equiv 3 (\mod 4)$ in this case, but we do not need this 
stronger conclusion.)
 \end{proof}

\begin{lem}\label{spor}
Theorem \ref{main:simple} holds true if $S$ is any of the $26$ sporadic simple groups.
\end{lem}

\begin{proof}
Since $p > 2$, $G = S$ in this case. Now if $P$ is non-cyclic, then one can check using \cite{W} that 
$|\bfN_G(P)|$ is even. Otherwise $P$ is cyclic as desired.
\end{proof} 
 
In what follows, we use the notation $SL^\eps$ to denote $SL$ when $\eps = +$ and $SU$ when $\eps = -$, and similarly
for $GL^\eps$, $PSL^\eps$, etc. We also use $E_6^\eps$ to denote $E_6$ when $\eps = +$ and $\tw2 E_6$ when $\eps = -$.
 
\begin{pro}\label{slu}
Theorem \ref{main:simple} holds if $S = PSL^\eps_n(q)$ with $n \geq 3$, $\eps = \pm$, and $p \nmid q$.
\end{pro}
 
\begin{proof}
(a) We view $S = L/\bfZ(L)$, where $L := SL^\eps_n(q) \nor H:= GL(V) \cong GL^\eps_n(q)$, and $V = \FF_q^n$ for 
$\eps = +$, $V = \FF_{q^2}^n$ for $\eps = -$. Fix a basis $(e_1, \ldots ,e_n)$ of $V$, which is orthonormal if $\eps = -$.
If $q = r^f$ for a prime $r$, then we can use this basis to define a field automorphism $\sigma$ of $H$, $L$, and $S$, which sends  a
matrix $(x_{ij}) \in H$ (written in the chosen basis) to $(x_{ij}^r)$, and let $\sigma_0$ denote the $p$-part of $\sigma$. 

We will use the description of a Sylow $p$-subgroup $R$ of $H$ as given in \cite[\S2]{Hall} using the results of
\cite[Chap. 4, \S10]{GL} and \cite[Chap. 4, \S4.10]{GLS}. In particular, $R = R_T \rtimes R_W$, with
$R_T$ (the ``toral part'' of $R$) being homocyclic abelian. Furthermore,
there is a direct sum (orthogonal if $\eps = -$) decomposition of $V$ (compatible with the chosen basis of $V$) as $R_T$-module:
$$V = V_0 \perp V_1 \perp  \ldots \perp V_m,$$
with $V_i \cong V_1$ for $1\leq i \leq m$ of dimension 
$e := \ord_p(\eps q)$,  and $0 \leq \dim V_0 < e$. Next, $R_T = R_1 \times \ldots \times R_m$
with $R_i$ a cyclic subgroup of a cyclic maximal torus $T_i$ of 
$GL^\eps(V_i)$ and $|T_i| = q^e-\eps^e$. Furthermore, there is a subgroup $\Sigma$ of $\bfN_H(R_T)$
with $\Sigma$ isomorphic to the symmetric group $\SSS_m$ acting naturally on the sets
$\{V_1,\ldots,V_m\}$, and $\{R_1, \ldots, R_m\}$, and
with $R_W$ (the ``Weyl part'' of $R$) being a Sylow $p$-subgroup of $\Sigma$. 

Since $\sigma_0$ normalizes $GL^\eps(V_1)$, we can embed $\sigma_0$ in a Sylow $p$-subgroup
$\tilde R_1$ of $GL^\eps(V_1) \rtimes \langle \sigma_0 \rangle$, and then take $R_1 = \tilde R_1 \cap GL^\eps(V_1)$ to ensure $R_1$ to be $\sigma_0$-stable. 
The subgroups $R_i$ are then constructed using the (fixed) isomorphism $V_i \cong V_1$ for $1 \leq i \leq m$. 
By its construction, $R$ is $\sigma_0$-stable. Then $LR$ is $\sigma_0$-stable and $p \nmid |H/LR|$. 
Since $p > 2$ and $G=SP \leq \Aut(S)$, by a suitable conjugation in $\Aut(S)$ we may assume that 
$G \leq H^*/\bfZ(H^*)$, where $H^* := LR \rtimes \langle \sigma_0 \rangle$ and 
$P \leq P^*/(P^* \cap \bfZ(H^*))$, where $P^* = R \rtimes \langle \sigma_0 \rangle$. Set $Q^* := L \cap P^*$.
 
\smallskip
(b) Here we consider the case $p|(q-\eps)$, whence $e = 1$ and $m = n$. Let $V_i$, $1 \leq i \leq m$, be spanned by a vector $e_i$
in such a way that $\Sigma$ acts on $\{e_1, e_2, \ldots ,e_m \}$.

\smallskip
(b1) Consider the case $n \geq p+2$. Then, by Lemma \ref{alt}(i), we can
find an involution $t \in \bfN_{\AAA_n}(R_W)$. By the choice, we have that $t \in L \smallsetminus \bfZ(H)$; furthermore,
$t$ normalizes $R_W$, $R_T$, $R$, and commutes with $\sigma_0$. Next, if $x \in R_1 < R_T$, then 
$txt^{-1}x^{-1} \in L \cap R_T \leq Q^*$. Since $P^* = Q^*R_1\langle \sigma_0\rangle$,
it follows that $t$ centralizes $P^*/Q^*$. We have shown that
$$tQ^* \in \bfC_{\bfN_L(Q^*)/Q^*}(P^*/Q^*).$$ 
Since $Q = Q^*/(Q^* \cap \bfZ(H^*)) \in \Syl_p(S)$, we get
$$tQ \in \bfC_{\bfN_S(Q)/Q}(P/Q) = \bfN_S(P)/Q.$$
Thus we have shown that $|\bfN_S(P)|$ is even, a contradiction. 

\smallskip
(b2) Next assume that $n=p+1$. Then we can choose $R_W$ to be generated by the $p$-cycle 
permuting $e_1, \ldots ,e_p$ cyclically and fixing $e_{p+1}$. By choosing $t(e_{p+1}) = \pm e_{p+1}$ suitably, 
we again get a $\sigma_0$-invariant monomial matrix $t \in \bfN_{L}(R_W) \smallsetminus \bfZ(H)$
of order $2$, and then repeat the arguments in (b1) to see that $|\bfN_S(P)|$ is even. 

If $3 \leq n < p$, then $R_W = 1$. Choosing $t = \diag(-1,-1, 1, \ldots 1)$ if $2 \nmid q$ and 
$$t:e_1 \leftrightarrow e_2,~e_i \mapsto e_i,~3 \leq i \leq n$$
if $2|q$, we then get a $\sigma_0$-invariant monomial
involution $t \in \bfN_L(R_W) \smallsetminus \bfZ(H)$ and finish as above.
 
\smallskip
(b3) It remains to consider the case $n=p$. Fix $\omega \in \FF_{q^2}^\times$ of order $p$ and consider
$y := \diag(\omega,\omega^2, \ldots,\omega^{p-1},1)$. We can also choose $R_W$ to be generated by the $p$-cycle 
permuting $e_1, \ldots ,e_p$ cyclically. Then one can check that $y \in \bfZ(Q)$. Note that $p|(q-\eps)$ implies 
that $\omega^{\sigma_0} = \omega$ and so $[y,\sigma_0] = 1$. Also, $y$ commutes with $T_1$. Since 
$P^* = Q^*R_1\langle \sigma_0\rangle$, we see that $y \in \bfZ(P) \cap Q$. On the other hand, $y$ is inverted by
the element
$$v: e_i \leftrightarrow e_{p-i},~1 \leq i \leq p-1,~e_p \mapsto (-1)^{(p-1)/2}e_p$$
of $L$, contradicting Lemma \ref{real}.

\smallskip
(c) From now on we may assume that $p \nmid q(q-\eps)$; in particular, $R < L$ and $e > 1$. 
Arguing as in (b), we see that it suffices to produce a $\sigma_0$-invariant involution
$t \in L \smallsetminus \bfZ(L)$ that induces an element in $\bfN_{\SSS_m}(R_W)$ and normalizes $R_T$.
Also note that Sylow $p$-subgroups of $S$ are cyclic if $m = 1$, so we may assume that $m \geq 2$. 
Now, if $2|q$, then $\bfN_{\Sigma}(R_W)$ is contained in $L$ (and $\sigma_0$-fixed pointwise), and so
we are done by Lemma \ref{alt}(ii). Similarly, if $2|e$, then any transposition in $\Sigma$ 
``flips'' two $e$-dimensional subspaces $V_i$ and $V_j$ and so has determinant $1$, whence we are again done by 
Lemma \ref{alt}(ii).

So we may assume $2 \nmid qe$. If $m \geq p+2$, then we can apply Lemma \ref{alt}(i). If $4 \leq m < p$,
then we can take $t \in \Sigma$ that represents $(12)(34)$ under $\Sigma \cong \SSS_m$. If $p > m = 2,3$, 
then take $s\in \Sigma$ that represents $(12)$ under $\Sigma \cong \SSS_m$ and set $t := (-1_{V_1})s$.
Finally, assume that $m = p$ or $p+1$. In this case, we may assume that $R_W$ is generated by the element
represented by the $p$-cycle $(1,2,\ldots,p)$ in $\SSS_m$, and choose an involution $s \in \bfN_{\Sigma}(R_W)$;
in particular, $s$ acts trivially on $V_{p+1}$ if $m > p$.
Then we can take $t = su^j$ with $u = -1_{V_1 \oplus V_2 \oplus \ldots \oplus V_p}$  and $j \in \{0,1\}$ chosen suitably.
 \end{proof}
 
\begin{pro}\label{ortho}
Theorem \ref{main:simple} holds if $S = P\Omega^\eps_{2n}(q)$ with $2 \nmid n \geq 5$, $\eps = \pm$, and $p \nmid q$.
\end{pro}
 
\begin{proof} 
(a) We view $S = L/\bfZ(L)$, where $L = \Omega(V) \cong \Omega^\eps_{2n}(q)$ and $V = \FF_q^{2n}$ is a quadratic space
of type $\eps$, and let $H = GO(V)$. Also define $\eps_p = +$ if $e=\ord_p(q)$ is odd and $\eps_p = -$ if $2|e$, and
let $d := {\mathrm {lcm}}(2,e)$. We again use the description of $R = R_T \rtimes R_W \in \Syl_p(H)$ as in part (a) of
the proof of Proposition \ref{slu}. The only differences are that $V_i \cong V_1$ for $1 \leq i \leq m$ is a quadratic space of 
dimension $d$ and type $\eps_p$, either $\dim V_0 < d$ or $\dim V_0 = d$ but $V_0$ has type $-\eps_p$, and 
$|T_1| = q^{d/2}-\eps_p$. 

We will choose a basis compatible with the decomposition $V = V_0 \perp V_1 \perp \ldots \perp V_m$
and the isomorphisms $V_i \cong V_1$ for $1 \leq i \leq m$, and use this basis to define the field automorphism 
$\sigma: (x_{ij}) \mapsto (x_{ij}^r)$ if $r$ is the prime dividing $q$. Let $\sigma_0$ denote the $p$-part of $\sigma$. As 
in the proof of Proposition \ref{slu}, we can construct $R$ in such a way that all $R_i$ and $R$ are $\sigma_0$-stable.  
As $p > 2$, we may identify $Q \in \Syl_p(S)$ with $R$. Furthermore, since $n \geq 5$
and $G=SP$, conjugating suitably in $\Aut(S)$, we may assume that $P \leq P^*:= R \rtimes \langle \sigma_0 \rangle$. 

\smallskip
(b) First we consider the case $V_0 \neq 0$. Since $2|d = \dim V_1$, we then have that $\dim V_0 \geq 2$. It follows 
that $R < \Omega^{\pm}_{2n-2}(q) < L$. Since $2|(n-1)$, the element $z$ obtained in Lemma \ref{real} is real in 
$\Omega^{\pm}_{2n-2}(q)$ by \cite[Proposition 3.1]{TZ} and so in $L$ as well, a contradiction.

Next suppose that $p \nmid m$ but $m > 1$. Then we can choose $R_W \in \Syl_p(\Sigma)$ 
to be contained in $\SSS_{m-1}$ under $\Sigma \cong \SSS_m$. As the $p$-group $\langle \sigma_0\rangle$ acts on 
$\Omega_1(R_m) \neq 1$, we can take $1 \neq z \in \Omega_1(R_m)$ 
to be $\sigma_0$-fixed, and then note that $z \in \bfZ(P) \cap R$. On the other hand, \cite[Proposition 3.1]{TZ} 
and the oddness of $n$ again imply that $z$ is real in $\Omega^\pm_{2n-2}(q) < L$, contradicting Lemma \ref{real}. 
 
\smallskip
(c) If $m  = 1$ then Sylow $p$-subgroups of $S$ are cyclic. From now on we may therefore assume that 
$V_0 = 0$, $m \geq 2$, and $p|m$. It suffices to produce a $\sigma_0$-invariant involution
$t \in L \smallsetminus \bfZ(L)$ that induces an element in $\bfN_{\SSS_m}(R_W)$ and normalizes $R_T$. Indeed, 
such an element would belong to $\bfN_L(P)$ and so $|\bfN_S(P)|$ would be even.

If $m \geq 2p$, then by Lemma \ref{alt}(i), we can choose
an involution $t \in \bfN_{\AAA_m}(R_W)$ which is then $\sigma_0$-fixed and belongs to $L \smallsetminus \bfZ(L)$,
and so we are done. 

So we may assume that $m = p$ and so $R_W \cong C_p$. We can find an involution $s \in \bfN_\Sigma(R_W)$ which is 
represented by a disjoint product of $(p-1)/2$ transpositions $s_i \in \SSS_p$, $1 \leq i \leq (p-1)/2$. Note that each $s_i$ flips two 
$d$-dimensional subspaces $V_k$ and $V_{k'}$ (and fixes all remaining $V_l$ pointwise), so $\det {s_i} = 1$ and $s \in SO(V)$. 

\smallskip
(d) Here we assume that $2 \nmid q$. Since all $V_i$ have type $\eps_p$ and $V$ has type $\eps$, we get $\eps = \eps_p$.
Next, $n = dm/2$ is odd, so $d \equiv 2n (\mod 4)$. If 
$-1_V \notin \Omega(V)$, then we can take $t = (-1_V)^js$ for a suitable $j \in \{0,1\}$ to get $t \in L \smallsetminus \bfZ(L)$ as desired.

Assume $-1_V \in \Omega(V)$. Then $\eps = (-1)^{n(q-1)/2}$ by \cite[Proposition 2.5.13(ii)]{KL}, and so $-1_{V_i} \in \Omega(V_i)$ 
since $d \equiv 2n (\mod 4)$ and $\eps = \eps_p$.  Consider for instance a flip $\tau:V_1 \leftrightarrow V_2$ in $\Sigma$ that sends 
an orthogonal basis $(e_1, \ldots ,e_d)$ of $V_1$ to a basis $(f_1, \ldots ,f_d)$ of $V_2$. Let $Q$ denote the quadratic form on $V$ and
let $\rho_v$ denote the reflection corresponding to any non-singular $v \in V$. As $-1_{V_i} = \prod^{d}_{i=1}\rho_{e_i} \in \Omega(V_i)$,
we have that $\prod^d_{i=1}Q(e_i) \in \FF_q^{\times 2}$. Now $\tau = \prod^d_{i=1}\rho_{e_i-f_i}$ and 
$$\prod^d_{i=1}Q(e_i-f_i) = \prod^d_{i=1}(Q(e_i)+Q(f_i)) = 2^d\prod^d_{i=1}Q(e_i) \in \FF_q^{\times 2}.$$
We have shown that $\tau \in \Omega(V)$. As a consequence, $s_i \in \Omega(V) = L$ for all $i$, and so we can just set 
$t := s = \prod_is_i$.

\smallskip
(e) Finally, let $2|q$. Again, we consider for a flip $\tau:V_1 \leftrightarrow V_2$ in $\Sigma$ that sends 
a basis $(e_1, \ldots ,e_d)$ of $V_1$ to a basis $(f_1, \ldots ,f_d)$ of $V_2$. Then $\tau$ fixes the maximal totally singular subspace 
$M := \langle e_1+f_1, \ldots ,e_d+f_d \rangle_{\FF_q}$ of $V_1 \oplus V_2$. It follows by \cite[Lemma 2.5.8]{KL} that 
$\tau \in \Omega(V_1 \oplus V_2) \leq L$. Hence, $s_i \in L$ for all $i$, and we are done by taking $t := s = \prod_is_i$. 
\end{proof}
 
The following simple argument is useful in various situations:

\begin{lem}\label{coprime}
Let $M= NP$ be a finite group with a normal subgroup $N$, $p$ be a prime, and $P \in \Syl_p(M)$. Let $A$ be a subgroup of $N$
that contains $Q := P \cap N$ as a normal subgroup. Suppose that $M$ acts transitively on the set of $N$-conjugates of $A$. 
Then, for any prime $r$, some Sylow $r$-subgroup $X$ of $\bfN_N(A)/Q$ is fixed by a Sylow $p$-subgroup of $M$ containing $Q$.
\end{lem}

\begin{proof}
Note that $Q \in \Syl_p(A)$ is normal in $A$  and so
$Q \nor \bfN_N(A)$. By assumption, $NP=M=\bfN_M(A)N$. Hence, without any loss, we may replace $(N,M)$ with
$(\bfN_M(A),\bfN_N(A))$ and assume that $A \lhd M$. Now the $p$-group $P$ acts on the set of Sylow $r$-subgroups of 
$N/Q$ which has $p'$-size, and so it must have a fixed point.  
\end{proof} 
 
\begin{pro}\label{e6}
Theorem \ref{main:simple} holds if $S$ is a simple group of type $E^\eps_6(q)$ with $\eps = \pm$ and $p \nmid q$.
\end{pro}
 
\begin{proof} 
(i) We can view $S$ as $[H,H]$ where $H := E^\eps_6(q)_{{\mathrm {ad}}}$. If $p \nmid (q^5-\eps)(q^9-\eps)$, then we 
can embed $Q := P \cap S$ in a subgroup $X \cong F_4(q)$ of $S$ and conclude that the element $z$ obtained in Lemma \ref{real} is real in 
$X$ by \cite[Proposition 3.1]{TZ}, a contradiction. On the other hand, if $p|(q^5-\eps)(q^9-\eps)$ but $p \nmid (q^3-\eps)$,
then Sylow $p$-subgroups of $S$ are cyclic. So we may assume that $p|(q^3-\eps)$.

Suppose that $p|(q^3-\eps)$ but $p \nmid (q-\eps)$; in particular, $(q,\eps) \neq (2,-)$. By the main result of \cite{LSS}
(see Table 5.2 therein), $H$ has a unique conjugacy class of maximal subgroups $A$ 
(of maximal rank) of type $T \cdot 3^{1+2} \cdot SL_2(3)$, where $T$ is a maximal torus of order 
$(q^2+\eps q +1)^3$. Now we can view $Q$ as $\bfO_p(T)$ and apply Lemma \ref{coprime} to $(M,N,A) = (GH,H,A)$ to 
conclude that $P$ fixes a Sylow $2$-subgroup $B \cong Q_8$ of $A/Q$ 
(note that $GH$ is a group as $G \leq \Aut(S)$ normalizes $H$ and 
that $\bfN_H(A) = A=\bfN_H(Q)$ by maximality of $A$). It then follows that $P$ fixes the subgroup $[B,B] \cong C_2$ 
which is contained in $\bfN_S(Q)/Q$. Thus $P$ fixes an involution in $\bfN_S(Q)/Q$ and so $|\bfN_G(P)|$ is even, a contradiction. 

Assume now that $p|(q-\eps)$ but $p \neq 3$ (whence $p \geq 5$ and $q \geq 4$). By the main result of \cite{LSS}
(see Table 5.1 therein), $H$ has a unique conjugacy class of maximal subgroups $A$ 
(of maximal rank) of type $C_d \cdot (PSL_2(q) \times PSL^\eps_6(q)) \cdot C_{de}$, where 
$d = \gcd(2,q-1)$ and $e = \gcd(3,q-\eps)$. By the Frattini argument, we may assume that $P$ normalizes $A$. Next, the 
assumption on $p$ implies that we can view $Q$ as a Sylow $p$-subgroup of $A$ and so contained 
in $C_d \cdot (PSL_2(q) \times PSL^\eps_6(q))$. As $P$ normalizes the component $A_1:= C_d \cdot PSL_2(q)$ of 
$[A,A]$, we may choose the element $z$ in Lemma \ref{real} to be contained in $A_1$. But then $z$ is real in $A_1$,
again a contradiction.

\smallskip
(ii) Finally, we consider the case $p = 3|(q-\eps)$. View $S = L/Z$ for $L = E_6^\eps(q)_{\mathrm {sc}} = \cG^F$ 
and $Z = \bfZ(L)$, where $\cG$ is a simple, simply connected algebraic group of type $E_6$ over $\bar{\FF}_q$ and 
$F:\cG \to \cG$ a Frobenius endomorphism. According to \cite[Theorem 25.11]{MT}, there is a unique $L$-conjugacy 
class of maximal tori $\cT$ in $\cG$ such that $T := \cT^F$ has order $(q-\eps)^6$, and 
$\bfN_L(\cT)/T = W(E_6)$ by \cite[Proposition 25.3]{MT}. Next, we can view $Q = \hat{Q}/Z$ for 
some $\hat{Q} \in \Syl_3(L)$.
By the Frattini argument, we may assume that $P$ normalizes $A:= \bfN_L(\cT)$ and $T = \sol(A)$; moreover, 
$\hat{Q} > \bfO_3(T)$ and $\bar{Q} := \hat{Q}T/T \in \Syl_3(W(E_6))$. Recall that $W(E_6) = SU_4(2) \rtimes C_2$.
Since $P$ normalizes $[A/T,A/T] \cong SU_4(2)$ and $\bar{Q}$, 
we see that $P$ normalizes the subgroups $C > B > T$ of $A$, where 
$$B/T = \bfN_{SU_4(2)}(\bar{Q}) = \bar{Q} \rtimes C_2,~~C/T = \bfN_{W(E_6)}(\bar{Q}) = \bar{Q} \rtimes C_2^2.$$
Setting $O := \bfO_{3'}(T)$, observe that $\bfN_{B}(\hat{Q})O/O = \bfN_{B/O}(\hat{Q}O/O)$ 
and so the previous equalities imply that 
$$2|\hat{Q}| \mbox{ divides }|\bfN_B(\hat{Q})| = |\bfN_C(\hat{Q})|/2.$$ 
According to \cite[Table 1]{MN}, $\bfN_L(\hat{Q}) = \hat{Q} \rtimes C_2^2$. As 
$\bfN_{C}(\hat{Q}) \leq \bfN_L(\hat{Q})$, we conclude that $\bfN_B(\hat{Q}) = \hat{Q} \rtimes C_2$, with the 
quotient $\bfN_B(\hat{Q})/\hat{Q} \cong C_2$ fixed by $P$. Since $\bfN_S(Q) = \bfN_L(\hat{Q})/Z$, we have therefore shown that 
$\bfN_S(Q)/Q$ contains a $P$-fixed involution. Thus $|\bfN_G(P)|$ is even, again a contradiction.
\end{proof} 
 
\begin{pro}\label{defi}
Theorem \ref{main:simple} holds if $S$ is a simple group of Lie type in characteristic $p$.
\end{pro}

\begin{proof}
As $p > 2$, we can view $S = L/Z$ for $L =  \cG^F$ and $Z = \bfZ(L)$, where $\cG$ is a simple, simply connected algebraic group 
over $\bar{\FF}_q$ and $F:\cG \to \cG$ a Steinberg endomorphism. 
Since $p \nmid |Z|$, we can identify $Q:= P \cap S$ with
a Sylow $p$-subgroup of $L$. It is well known that $\bfN_L(Q) = QT = \cB^F$, where $T = \cT^F$, 
$\cT$ is an $F$-stable maximal torus of $\cG$ contained in an $F$-stable Borel subgroup $\cB$ of 
$\cG$, see e.g. \cite[\S2.3]{GLS}. As $\cG$ is not of type $D_4$ by Corollary \ref{simple1} and $p > 2$, it follows that the subgroup 
of field automorphisms in $\Out(S)$ contains a Sylow $p$-subgroup of $\Out(S)$. Hence, we may assume
(using the uniqueness of $(\cB,\cT)$ up to conjugacy) that $G \leq H:= L \rtimes \langle \sigma \rangle$ for a suitable field 
automorphism $\sigma$ and $\sigma$ acts on $T$. Any $2$-element in $T^\sigma$ will then belong to $\bfC_{\bfN_L(Q)/Q}(P)$;
also, $\bfN_S(Q)/Q = (\bfN_L(Q)/Q)/(QZ/Q)$ and $QZ/Q \cong Z$. 
Hence, $|\bfN_G(P)|$ is even whenever the $2$-rank $a$ of $T^\sigma$ is larger than the $2$-rank $b$ of  $Z$. 

Note that $T^\sigma = \cT^{F'}$ for a suitable Frobenius endomorphism $F'$ of $\cG$. Now, if $\cG$ is of type $D_n$ with $n \geq 4$,
then $a \geq 3$ and $b \leq 2$. In all other cases $Z$ is cyclic (see eg. \cite[Theorem 2.5.12]{GLS})
and so $b \leq 1$. So we only need to consider the cases where
$a \leq b \leq 1$, whence $L$ is of type $SL_2$ (note that $(a,b) = (1,0)$ for $L$ of type $SU_3$ or $\tw2 G_2$). Finally, 
when $L = SL_2(q)$ we must have $q \equiv 3 (\mod 4)$ by Corollary \ref{simple1}(ii).
\end{proof} 
 
Theorem \ref{main:simple} now follows from Corollary \ref{simple1},  Lemmas \ref{alt}, \ref{spor}, and Propositions \ref{slu}, \ref{ortho}, \ref{e6},
and \ref{defi}.
\hfill $\Box$

\section{Odd Sylow Normalizers and Character Tables}

In this section we prove several reflections (of elementary nature)
of the existence of odd Sylow normalizers, such as Theorem C,
in the character table.

\begin{thm}\label{ker}
Let $G$ be a finite group and let $p$ be a prime, let  $P \in \syl pG$, and assume that
$|\norm GP|$ is odd. If $\chi \in \irrp G$ is real-valued, then $N \leq \ker\chi$ for every
solvable $N \nor G$.
\end{thm}

\begin{proof}
Let $\chi \in \irrp G$. Since $\chi$ has $p'$-degree, it follows that
$\chi_N$ has a $P$-invariant constituent $\theta$. By the Frattini argument,
two such $P$-invariant constituents are $\norm GP$-conjugate.
Now, since $\chi$ is real-valued, it follows that $\bar\theta$ also
lies under $\chi$. Hence, there exists $g \in \norm GP$ such that
$\theta^g=\bar\theta$. Then $g^2$ fixes $\theta$. Since $\norm GP$ has odd order,
it follows that $g$ fixes $\theta$, and therefore, we have that $\theta$ is
a real-valued $P$-invariant character. Now, by Lemma 2.1 of \cite{NS},
we have that $\theta$ has a real extension $\eta$ to $NP$. 
Now, $NP$ is a solvable group with a Sylow $p$-subgroup $P$
and having odd Sylow normalizer. By Theorem A of \cite{IMN1},
we have that $\eta$ is principal, and therefore $\theta$ is principal.
\end{proof}

If  $G$
  has no non-trivial real conjugacy classes of $p'$-size, then it is not true that
 this condition is inherited by normal subgroups.
 Take $G$ the semidirect product of the extra-special 3-group $F$
 of exponent 3 acted on by an involution that inverts $F/\zent F$
 and centralizes $\zent F$. Set $p=3$. Then $G/Z$ possesses a
 real class of size 3. 
 
 \medskip

 \begin{thm}
 Suppose that $G$ is a finite group,
 let $p$ be an odd prime, and let $P \in \syl pG$.
 If
 $|\norm GP|$ is odd, then  for every $N \nor G$ the only
 real class of $G/N$ with $p'$-size is the trivial class.
 \end{thm}
 
 \medskip
 
 \begin{proof}~~Suppose that $|\norm GP|$ is odd.
 If $N \nor G$, then $\norm{G/N}{PN/N}=\norm GP N/N \cong \norm GP/\norm NP$
 also has odd order. So assume that $x^G$ is a real class of $p'$-size.
 Then we may assume that there is
 some $t \in G$ such that $x^t=x^{-1}$ and $P \leq \cent Gx$.
 Then $P^t \leq \cent Gx$ and $P^t=P^v$ for some $v \in \cent Gx$.
 Then $n=tv^{-1} \in \norm GP$ and $x^n=x^{-1}$.
 Thus $x^{n^2}=x$. Since $\langle n \rangle=\langle n^2\rangle$,
 we conclude that $x=x^{-1}$ and so $x^2=1$. Since $P \leq \cent Gx$
 we have that $x \in \cent GP \leq \norm GP$. Thus $x=1$, because $\norm GP$
 has odd order.
 \end{proof}
 
 If $\norm GP$ has odd order, there is a deeper necessary condition that can be read off
 the character table (assuming the Galois version of the McKay conjecture \cite{N}): {\it
 If $\sigma$ is the Galois automorphism that complex-conjugates $p$-power roots
 of unity and fixes $p'$-roots of unity, then
 every $\sigma$-fixed $p'$-degree irreducible character of $G$ is $p$-rational}.

\end{document}